\documentclass[12pt]{amsart}
\usepackage{graphicx}
\usepackage[centertags]{amsmath}
\usepackage{amsfonts}
\usepackage{amssymb}
\usepackage{amsthm}
\usepackage{newlfont}
\newtheorem{thm}{Theorem}[section]
\newtheorem{cor}[thm]{Corollary}
\newtheorem{lem}[thm]{Lemma}

\theoremstyle{definition}

\theoremstyle{remark}

\newcommand{\grad}{\textrm{grad}}

\title{Almost conformal transformation in a four dimensional Riemannian manifold with an additional structure}
\bigskip
\author{Iva Dokuzova}
\date{}
\begin{document}
\footnotetext[1]{This work is partially supported by project RS11
- FMI - 004 of the Scientific Research Fund, Paisii Hilendarski
University of Plovdiv, Bulgaria} \maketitle

\begin{abstract}
We consider a $4$-dimensional Riemannian manifold $M$ with a
metric $g$ and affinor structure $q$. The local coordinates of
these tensors are circulant matrices. Their first orders are $(A,
B, C, B)$, $A, B, C\in FM$ and $(0, 1, 0, 0)$, respectively.

We construct another metric $\tilde{g}$ on $M$.  We find the
conditions for $\tilde{g}$ to be a positively defined metric, and
for $q$ to be a parallel structure with respect to the Riemannian
connection of $g$.

Further, let $x$ be an arbitrary vector in $T_{p}M$, where $p$ is
a point on $M$. Let $\varphi$ and $\phi$ be the angles between $x$
and $qx$, $x$ and $q^{2}x$ with respect to $g$. We express the
angles between $x$ and $qx$, $x$ and $q^{2}x$ with respect to
$\tilde{g}$ with the help of the angles $\varphi$ and $\phi$.

Also,we construct two series $\{\varphi_{n}\}$ and $\{\phi_{n}\}$.
We prove that every of it is an increasing one and it is converge.
\end{abstract}
\Small{\textbf{Mathematics Subject Classification (2010)}: 53C15,
53B20}

\Small{\textbf{Keywords}: Riemannian metric, affinor structure}

\section{Introduction}
\thispagestyle{empty}
 The main purpose of the present paper is to
continue the investigations in \cite{1}, \cite{2}, \cite{3}. We
study a class of Riemannian manifolds which admits a circulant
metric $g$ and an additional circulant structure $q$. The forth
degree of structure $q$ is an identity, and $q$ is a parallel
structure with respect to the Riemannian connection $\nabla$ of
$g$.

\section{Preliminaries}
 We consider a $4$-dimensional Riemannian manifold $M$ with a
metric $g$ and an affinor structure $q$. We note the local
coordinates of $g$ and $q$ are circulant matrices. The next
conditions and results have been discussed in \cite{3}.

The metric $g$ have the coordinates:
\begin{equation}\label{1}
    g_{ij}=\begin{pmatrix}
      A & B & C & B \\
      B & A & B & C\\
      C & B & A & B\\
      B & C & B & A\\
    \end{pmatrix},\qquad A > C > B >0
\end{equation}
in the local coordinate system $(x_{1}, x_{2}, x_{3}, x_{4})$, and
$A=A(p), B=B(p), C=C(p)$, where $p(x_{1}, x_{2}, x_{3}, x_{4})\in
F\subset R^{4}$.  Naturally, $A, B, C$ are smooth functions of a
point $p$. We calculate that $det
g_{ij}=(A-C)^{2}((A+C)^{2}-4B^{2})\neq 0$.

Further, let the local coordinates of $q$ be
\begin{equation}\label{3}
    q_{i}^{.j}=\begin{pmatrix}
      0 & 1 & 0 & 0\\
      0 & 0 & 1 & 0\\
      0 & 0 & 0 & 1\\
      1 & 0 & 0 & 0\\
    \end{pmatrix}.
\end{equation}

We will use the notation $\Phi_{i}=\dfrac{\partial \Phi}{\partial
x^{i}}$ for every smooth function $\Phi$ defined in $F$.

We know from \cite{3} that the following identities are true
\begin{equation}\label{4}
    q^{4}=E;\quad q^{2}\neq \pm E;
\end{equation}
\begin{equation}\label{5}
    g(qw, qv)=g(w, v),\quad w, v \in \chi M,
\end{equation}
where $E$ is the unit matrix;
\begin{equation}\label{5*}
    0 < B < C < A \quad \Rightarrow \quad g \ is \ possitively \ defined.
\end{equation}
Now, let $w=(x, y, z, u)$ be a vector in $\chi M$. Using (\ref{1})
and (\ref{3}) we calculate that

\begin{equation}\label{6} g(w, w)=
A(x^{2}+y^{2}+z^{2}+u^{2})+2B(xy+xu+yz+zu)+2C(xz+yu)\end{equation}
\begin{equation}\label{7} g(w, qw)= (A+C)(xu+xy+yz+zu)+B(x^{2}+ y^{2}+z^{2}+u^{2}+2xz+2yu)\end{equation}
\begin{equation}\label{8} g(w, q^{2}w)= 2A(xz+yu)+2B(xu+
xy+zy+zu)+C(x^{2}+y^{2}+z^{2}+u^{2}).\end{equation}

Let $M$ be the Riemannian manifold with a metric $g$ and an
affinor structure $q$, defined by (\ref{1}) and (\ref{3}),
respectively. Let $w(x, y, z, u)$ be no eigenvector on $T_{p}M$(
i.e. $w(x, y, z, u)\neq (x, x, x, x)$, $w(x, y, z, u)\neq (x, -x,
x, -x)$ ). If $\varphi$ is the angle between $x$ and $qx$, and
$\phi$ is the angle between $x$ and $q^{2}x$, then we have
$\cos\varphi=\dfrac{g(w, qw)}{g(w,w)},$ $\cos\phi=\dfrac{g(w,
q^{2}w)}{g(w,w)},$ $\varphi\in (0,\pi)$, $\phi\in (0,\pi)$.

 We apply (\ref{6}), (\ref{7}) and (\ref{8}) in the above equations and we
get
\begin{equation}\label{9}
    \cos\varphi=\frac{(A+C)(xu+xy+yz+zu)+B(x^{2}+
    y^{2}+z^{2}+u^{2}+2xz+2yu)}{A(x^{2}+y^{2}+z^{2}+u^{2})+2B(xy+xu+yz+zu)+2C(xz+yu))},
\end{equation}
\begin{equation}\label{10}
    \cos\phi=\frac{C(x^{2}+y^{2}+z^{2}+u^{2})+2B(xy+xu+yz+zu)+2A(xz+yu)}{A(x^{2}+y^{2}+z^{2}+u^{2})+2B(xy+xu+yz+zu)+2C(xz+yu))}.
\end{equation}


\section{Almost conformal transformation in $M$}

Let $M$ satisfies (\ref{1})-- (\ref{5*}). We note
$f_{ij}=g_{ik}q_{t}^{k}q_{j}^{t}$, i.e.
\begin{equation}\label{11}
f_{ij}=\begin{pmatrix}
      C & B & A & B\\
      B & C & B & A\\
      A & B & C & B\\
      B & A & B & C\\
    \end{pmatrix}.
\end{equation}
We calculate $det f_{ij}=(C-A)^{2}((A+C)^{2}-4B^{2})\neq 0$, so we
accept $f_{ij}$ for local coordinates of another metric $f$. The
metric $f_{ij}$ is necessarily undefined. Further, we suppose
$\alpha$ and $\beta$ are two smooth functions in $F\subset R^{4}$
and we construct the metric $\tilde{g}$, as follows:
\begin{equation}\label{12}
    \tilde{g}=\alpha .g +\beta .f.
\end{equation}
We say that equation (\ref{12}) define an almost conformal
transformation, noting that if $\beta =0$ then (\ref{12}) implies
the case of the classical conformal transformation in $M$
\cite{2}.

From (\ref{1}), (\ref{3}), (\ref{11}) and (\ref{12}) we get the
local coordinates of $\tilde{g}$:
\begin{equation}\label{13}
\tilde{g}_{ij}=\begin{pmatrix}
      \alpha A+\beta C & (\alpha+\beta)B & \alpha C+\beta A & (\alpha+\beta)B  \\
      (\alpha+\beta)B  & \alpha A+\beta C & (\alpha+\beta)B & \alpha C+\beta A \\
      \alpha C+\beta A & (\alpha+\beta)B & \alpha A+\beta C & (\alpha+\beta)B \\
      (\alpha+\beta)B & \alpha C+\beta A & (\alpha+\beta)B & \alpha A+\beta
      C \\
    \end{pmatrix}.
\end{equation}
We see that $f_{ij}$ and $\tilde{g}_{ij}$ are both circulant
matrices.
\begin{thm}\cite{3}
Let $M$ be a Riemannian manifold with a metric $g$ from (\ref{1})
and an affinor structure $q$ from (\ref{3}). Let $\nabla$ be the
Riemannian connection of $g$. Then $\nabla q=0$ if and only if,
when
\begin{equation}\label{13*}
    \grad A=(\grad C)q^{2};\quad 2\grad B= (\grad C)(q+q^{3}).
\end{equation}
\end{thm}
\begin{thm}
Let $M$ be a Riemannian manifold with a metric $g$ from (\ref{1})
and an affinor structure $q$ from (\ref{3}). Also, let $\tilde{g}$
be a metric of $M$, defined by (\ref{12}). Let $\nabla$ and
$\tilde{\nabla}$ be the corresponding connections of $g$ and
$\tilde{g}$, and $\nabla q=0$. Then $\tilde{\nabla}q=0$ if and
only if, when
\begin{equation}\label{14}
    \grad \alpha=\grad\beta.q^{2}; \quad
    \grad\beta=-\grad\beta.q^{2}.
\end{equation}
\end{thm}
\begin{proof}
At first we suppose (\ref{14}) is valid. Using (\ref{14}) and
(\ref{13*}) we can verify that the following identity is true:
\begin{equation}\label{15}
    \grad (\alpha A+\beta C)=\grad (\alpha C+\beta A).q^{2},\ 2\grad (\alpha +\beta)B=\grad (\alpha C+\beta
    A).(q+q^{3})
\end{equation}
The identity (\ref{15}) is analogue to (\ref{13*}), and
consequently we conclude $\tilde{\nabla}q=0$.

Inversely, if $\tilde{\nabla}q=0$ then analogously to (\ref{13*})
we have (\ref{15}). Now, (\ref{13*}) and (\ref{15}) imply the
system
\begin{equation}\label{ur1}
    A\grad\alpha+C\grad\beta=(C\grad\alpha+A\grad\beta)q^{2}
    \end{equation}
\begin{equation}\label{ur2}
    2B(\grad\alpha+\grad\beta)=
    (C\grad\alpha+A\grad\beta)(q+q^{3}).
\end{equation}
From (\ref{ur1}) we find the only solution $\grad
\alpha=\grad\beta.q^{2}$, and from (\ref{ur2}) we get the only
solution $\grad \beta=-\grad\beta.q^{2}$. So the theorem is
proved.
\end{proof}
\begin{lem}\label{16}
    Let $\tilde{g}$ be the metric given by (\ref{12}). If $0 < \beta
    <\alpha$ and $g$ is positively defined, then $\tilde{g}$ is also
    positively defined.
    \end{lem}
    \begin{proof}
From the condition
  $(\alpha-\beta)(A - C) >
    0$ we get $\alpha A+ \beta C > \beta A +\alpha C > 0$. Also, we see that $\beta A +\alpha C >(\alpha+\beta)B > 0$
    and finely
  $(\alpha A+ \beta C)> \beta A +\alpha C >(\alpha+\beta)B > 0$.  Analogously to
    (\ref{5*}) we state that $\tilde{g}$ is positively defined.
    \end{proof}
\begin{lem}\label{17}
    Let $w=w(x(p), y(p), z(p), u(p))$ be in $T_{p}M$,
$p\in M$, $qw\neq w$, $q^{2}w\neq w$ and $g$ and $\tilde{g}$ be
the metrics of $M$, related by (\ref{12}). Then we have:

\begin{equation*}
\begin{split}
\tilde{g}(w, w)=(\alpha A+\beta
C)(x^{2}+y^{2}+z^{2}+u^{2})+2(\alpha+\beta)B(xy+xu+yz+zu)+\\2(\alpha
C+\beta A)(yu+xz) \end{split}
\end{equation*}
\begin{equation}\label{191}
\begin{split}
 \tilde{g}(w,
qw)=(\alpha+\beta)(A+C)(xu+xy+yz+zu)+\\
(\alpha+\beta)B(x^{2}+ y^{2}+z^{2}+u^{2}+2xz+2yu)
\end{split}
\end{equation}
\begin{equation*}
\begin{split}
\tilde{g}(w, q^{2}w)= 2(\alpha A+\beta
C)(xz+yu)+2(\alpha+\beta)B(xu+ xy+zy+zu)\\+(\alpha C+\beta
A)(x^{2}+y^{2}+z^{2}+u^{2}).
\end{split}
\end{equation*}
    \end{lem}

    \begin{thm}\label{t4}
    Let $w=w(x(p), y(p),z(p), u(p))$ be a vector in $T_{p}M$,
$p\in M$, $qw\neq w$, $q^{2}w\neq w$. Let $g$ and $\tilde{g}$ be
two positively defined metrics of $M$, related by (\ref{12}). If
$\varphi$ and $\varphi_{1}$ are the angles between $w$ and $qw$,
with respect to $g$ and $\tilde{g}$, $\phi$ and $\phi_{1}$ are the
angles between $w$ and $q^{2}w$, with respect to $g$ and
$\tilde{g}$, then the following equations are true:
\begin{equation}\label{18}
\cos\varphi_{1}=\frac{(\alpha +\beta)\cos \varphi}{\alpha+\beta
\cos\phi},
\end{equation}
\begin{equation}\label{19}
\cos\phi_{1}=\frac{\alpha \cos \phi+\beta}{\alpha+\beta \cos\phi}.
\end{equation}
    \end{thm}
\begin{proof}
    Since $g$ and $\tilde{g}$ are both positively defined metrics we
    can calculate $\cos \varphi$ and $\cos\varphi_{1}$, respectively. Then by
    using (\ref{13}) and (\ref{191}) we get
    (\ref{18}). Also, we calculate $\cos \phi$ and $\cos\phi_{1}$, respectively. Then by
    using (\ref{13}) and (\ref{191}) we get
    (\ref{19}).
    \end{proof}
Theorem \ref{t4} implies immediately the assertions:

\begin{cor}
Let $\varphi$ and $\varphi_{1}$ be the angles between $w$ and $qw$
with respect to $g$ and $\tilde{g}$. Let $\phi$ and $\phi_{1}$ be
the angles between $w$ and $q^{2}w$ with respect to $g$ and
$\tilde{g}$ Then

1)$\varphi= \dfrac{\pi}{2}$ if and only if when
$\varphi_{1}=\dfrac{\pi}{2}$ ;

2)if $\phi= \dfrac{\pi}{2}$ then
$\phi_{1}=\arccos\dfrac{\beta}{\alpha}$

3)if $\phi_{1}= \dfrac{\pi}{2}$ then $\phi=
\arccos(-\dfrac{\beta}{\alpha})$.
\end{cor}
Further, we consider an infinite series of the metrics of $M$ as
follows:
\begin{equation*}
    g_{0}, \ g_{1}, \ g_{2},\dots, \ g_{n}, \dots
\end{equation*}
where
\begin{equation}\label{20}
\begin{split}
       g_{0}=g,\quad g_{1}=\tilde{g}, \quad g_{n}=\alpha g_{n-1}+\beta f_{n-1},\\\quad f_{n-1,is}= g_{n-1, ka}q_{s}^{a}q_{i}
   ^{k},\quad   0 < \beta < \alpha .
   \end{split}
   \end{equation}
   By the method of the mathematical induction we can see that the
matrix of every $g_{n}$ is circulant one and every $g_{n}$ is
positively defined.
\begin{thm}\label{5}
 Let $M$ be a Riemannian manifold with metrics $g_{n}$ from (\ref{20})
and an affinor structure $q$ from (\ref{3}). Let $w=w(x(p), y(p),
z(p))$ be in $T_{p}M$, $p\in M$, $qw\neq w$, $q^{2}w\neq w$. Let
$\varphi_{n}$ be the angle between $w$ and $qw$, with respect to
$g_{n}$, let $\phi_{n}$ be the angle between $w$ and $q^{2}w$ with
respect to $g_{n}$. Then the infinite series:
\begin{equation*}
 1) \qquad  \varphi_{0}, \ \varphi_{1}, \ \varphi_{2},\dots, \ \varphi_{n}, \dots
\end{equation*}
is converge and $\lim \varphi_{n}=0$,
\begin{equation*}
  2)\qquad  \phi_{0}, \ \phi_{1}, \ \phi_{2},\dots, \ \phi_{n}, \dots
\end{equation*}
is converge and $\lim \phi_{n}=0$.
    \end{thm}
\begin{proof}
    Using the method of the mathematical induction and Theorem \ref{4} we
    obtain:
    \begin{equation}\label{21*}
\cos\varphi_{n}=\frac{(\alpha+\beta)\cos
\varphi_{n-1}}{\alpha+\beta \cos\phi_{n-1}}
\end{equation}
as well as $\varphi_{n}\in (0, \pi)$. From (\ref{21*}) we get:
\begin{equation}\label{22*}
\frac{\cos\varphi_{n}}{\cos\varphi_{n-1}}=\frac{\alpha+\beta}{\alpha+\beta
\cos\varphi_{n-1}}\geq 1.
\end{equation}
The equation (\ref{22*}) implies
$\cos\varphi_{n}\geq\cos\varphi_{n-1}$, so the series
$\{\cos\varphi_{n}\}$ is increasing one and since $\cos\varphi_{n}
< 1$ then it is converge. From (\ref{21*}) we have
$\lim\cos\varphi_{n}=1$, so $\lim\varphi_{n}=0$.

Now, we find
\begin{equation}\label{21}
\cos\phi_{n}=\frac{\alpha\cos \phi_{n-1}+\beta}{\alpha+\beta
\cos\phi_{n-1}}
\end{equation}
as well as $\phi_{n}\in (0, \pi)$. From (\ref{21}) we get:
\begin{equation}\label{22}
\cos\phi_{n}-\cos\phi_{n-1}=\frac{\beta\sin^{2}\phi_{n-1}}{\alpha+\beta
\cos\phi_{n-1}}\geq 0.
\end{equation}
The equation (\ref{22}) implies $\cos\phi_{n}>\cos\phi_{n-1}$, so
the series $\{\cos\phi_{n}\}$ is increasing one and since
$\cos\phi_{n}< 1$ then it is converge. From (\ref{21}) we have
$\lim\cos\phi_{n}=1$, so $\lim\phi_{n}=0$.
    \end{proof}

\author{Iva Dokuzova\\ University of Plovdiv\\Faculty
of Mathematics and Informatics\\ Department of
Geometry\\236 Bulgaria Blvd.\\ Bulgaria 4003\\
e-mail:dokuzova@uni-plovdiv.bg}\\

\end{document}